\documentclass[11pt]{amsart}
\usepackage{amsmath}
\usepackage{bbm}
\usepackage{mathrsfs}
\usepackage[english]{babel}
\usepackage[utf8]{inputenc}
\usepackage[T1]{fontenc}
\usepackage{wrapfig}

\usepackage{ dsfont }

\usepackage[top=3cm,bottom=2cm,left=3cm,right=3cm,marginparwidth=1.75cm]{geometry}
\usepackage{amsmath,amssymb}
\usepackage{graphicx}
\usepackage[colorinlistoftodos]{todonotes}
\usepackage[colorlinks=true, allcolors=blue]{hyperref}
\usepackage{cite}
\usepackage{enumitem} 
\usepackage[all,cmtip]{xy}
\usepackage{tikz-cd}
\usepackage{lmodern}
\usepackage{pgfplots}
\pgfplotsset{compat=1.17}
\usetikzlibrary{intersections}
\usepackage{relsize}
\usepackage{subcaption}
\usepackage{dsfont}

\usepackage{xcolor, import}
\graphicspath{{Figures/}}

\newtheorem{theorem}{Theorem}[section]
\newtheorem{proposition}[theorem]{Proposition}
\newtheorem{corollary}[theorem]{Corollary}
\newtheorem{lemma}[theorem]{Lemma}

\newtheorem{definition}[theorem]{Definition}

\theoremstyle{remark} 
\newtheorem{remark}[theorem]{Remark}
\newtheorem{example}[]{Example}

\usepackage{mathtools} 
\mathtoolsset{showonlyrefs,showmanualtags}
\numberwithin{equation}{section}
\usepackage{amsmath}

\newcommand{\R}{\mathbb{R}}

\newcommand{\ZZo}{\mathscr{Z}_0}
\newcommand{\KK}{\mathscr{K}}

\newcommand{\dd}{\mathrm{d}}

\newcommand{\EE}{\mathbb{E}}

\newcommand{\set}[2]{\left\{ #1 \ | \ #2\right\}}

\newcommand{\puff}[2]{\text{puff}_{#1} \left(#2\right)}
\newcommand{\Sp}[1]{\Sigma_\pi #1}
\newcommand{\MSp}{\mathrm{M}\Sigma_\pi}

\newcommand\restr[2]{{
  \left.\kern-\nulldelimiterspace 
  #1 
  \vphantom{\big|} 
  \right|_{#2} 
  }}

\DeclareMathOperator{\Vol}{vol}

\DeclareMathOperator{\conv}{conv}

\newcommand{\seg}{\underline}

\title{Fiber convex bodies}
\author{Léo Mathis \and Chiara Meroni}
\date{}

\begin{document}

\maketitle
\begin{abstract}
    In this paper we study the fiber bodies, that is the extension of the notion of fiber polytopes for more general convex bodies. After giving an overview of the properties of the fiber bodies, we focus on three particular classes of convex bodies. First we describe the strict convexity of the fiber bodies of the so called puffed polytopes. Then we provide an explicit equation for the support function of the fiber bodies of some smooth convex bodies. Finally we give a formula that allows to compute the fiber bodies of a zonoid with a particular focus on certain zonoids called discotopes. Throughout the paper we illustrate our results with detailed examples.
\end{abstract}
\section{Introduction}
If $K$ is a convex body in $\R^{n+m}$ and $\pi:\R^{n+m}\to V$ is the orthogonal projection onto a subspace $V\subset\R^{n+m}$ of dimension $n$, the \emph{fiber body} of $K$ with respect to $\pi$ is the \emph{average} of the fibers of $K$ under this projection:
\begin{equation}\label{eq:intro}
	\Sigma_\pi K= \int_{\pi(K)} \left(K\cap\pi^{-1}(x)\right) \dd x.
\end{equation}
This expression will be made rigorous in Proposition~\ref{prop:supportaverage}.

Such a notion was introduced for polytopes by  Billera and Sturmfels in \cite{fiberpolytopes}. It has been investigated in many different contexts, from combinatorics such as in \cite{athanasiadis2000fiber} to algebraic geometry and even tropical geometry in the context of polynomial systems \cite{esterovkhovanskii, esterovmixedfiber, sturmfels2008tropical}. Notably, recent studies concern the particular case of monotone path polytopes \cite{black2021monotone}.

This paper is dedicated to the study of the fiber body of convex bodies that are not polytopes. This construction was introduced and studied by Esterov in~\cite{esterovmixedfiber}. In Section~\ref{sec:Generalities} the general properties of fiber bodies are stated. In particular, we show in Example~\ref{ex:counterex_continuity} that a point of the boundary of the fiber body may not have a continuous representative. In the rest of the paper, each section regards the fiber body of a particular class of convex bodies.

Section~\ref{sec:Puffed} applies directly the description of the faces to certain convex bodies that we call \emph{puffed polytopes}. They are convex bodies that are obtained from polytopes by taking the ``derivative'' of their algebraic boundary (see Definition~\ref{def:puffed}). Propositions~\ref{prop:puff1}, \ref{prop:puff2} and \ref{prop:puffn} describe the strict convexity of the fiber body of a puffed polytope. As a concrete example we study the case of the elliptope with a particular projection.

In Section~\ref{sec:Smooth} we investigate the class of curved convex bodies. Namely, we consider convex bodies whose boundary are $C^2$ hypersurface with no ``flat'' directions, i.e. with a strictly positive curvature. In that case Theorem~\ref{thm:supofsmooth} gives an explicit formula for the support function of $\Sigma_\pi K$, directly in terms of the support function of $K$. This is an improvement of equation~\eqref{eq:supportaverage} which involves the support function of the fibers. We immediately give an example in which the support function of the fiber body is easily computed using Theorem~\ref{thm:supofsmooth}.

The last section is dedicated to the case of zonoids. Zonoids arise as limits of finite Minkowski sums of segments. We prove that the fiber body of a zonoid is a zonoid, and give an explicit formula to compute it in Theorem~\ref{thm:Fiberofzonoids}. We then focus on a particular class of zonoids that are finite Minkowski sums of discs in $3$--space, called \emph{discotopes}. After giving a general description of discotopes as algebraic bodies, we illustrate our formula for zonoids by computing the fiber body of a specific discotope.

\subsection*{Acknowledgments}

The authors wish to thank Antonio Lerario and Bernd Sturmfels without whom this project would not have existed, and Rainer Sinn for his helpful comments. We want to thank also Fulvio Gesmundo for interesting discussions and Anna-Laura Sattelberger for sharing her knowledge on holonomicity. We are also grateful to Michele Stecconi whose comments helped to considerably simplify Section~\ref{sec:Generalities}.

\subsection*{Data availability} Data sharing not applicable to this article as no datasets were generated or analysed during the current study.

\section{Generalities}\label{sec:Generalities}

\subsection{Main definitions}

Consider the Euclidean vector space $\R^{n+m}$ endowed with the standard Euclidean structure and let $V\subset \R^{n+m}$ be a subspace of dimensions $n$. Denote by $W$ its orthogonal complement, such that $\R^{n+m}=V\oplus W$.
Let $\pi : \R^{n+m}\to V$ be the orthogonal projection onto $V$.
Throughout this article we will canonically identify the Euclidean space with its dual. However the notation is meant to be consistent: $x,y,z$ will denote vectors, whereas we will use $u,v,w$ for dual vectors.

We call \emph{convex bodies} the non--empty compact convex subsets of a vector space. The space of convex bodies in a vector space $E$ is denoted by $\KK(E)$. If $K,L\in\KK(E)$ their \emph{Minkowski sum} is the convex body $K+L\in\KK(E)$ given by
\begin{equation}
    K+L:=\set{x_1+x_2}{x_1\in K,\ x_2\in L}.
\end{equation}
Moreover if $\lambda\in \R$, we write $\lambda K:=\set{\lambda x}{x\in K}.$

The \emph{support function} of a convex body $K\in\KK(\R^{n+m})$ is the function $h_K:\R^{n+m}\to \R$ given for all $u\in \R^{n+m}$ by 
\begin{equation}\label{eq:defsupp}
    h_K(u):=\max\set{\langle u, x\rangle}{x\in K},
\end{equation}
where $\langle \cdot , \cdot \rangle$ is the standard Euclidean scalar product.
This map becomes handy when manipulating convex bodies as it satisfies some useful properties (see~\cite[Section~$1.7.1$]{bible} for proofs and more details).
\begin{proposition}\label{prop:propertieshK}Let $K,L\in\KK(\R^{n+m})$ with their respective support functions $h_K,h_L$. Then
\begin{enumerate}[label=(\roman*)]
    \item $h_K=h_L$ if and only if $K=L$;
    \item If $T:\R^{n+m}\to \R^k$ is a linear map then $h_{TK}=h_K\circ  T^{t}$;
    \item $h_K$ is differentiable at $u\in\R^{n+m}$ if and only if the point $x$ realizing the maximum in~\eqref{eq:defsupp} is unique. In that case $x=\nabla h (u)$ where $\nabla h$ denotes the gradient of $h.$
\end{enumerate}
\end{proposition}

If $K\in \KK(\R^{n+m})$ we write $K_x$ for the orthogonal projection onto $W$ of the fiber of $\pi|_K$ over $x$, namely
\begin{equation}
    K_x := \set{y \in W}{(x,y)\in K}.
\end{equation}

\begin{definition}
    A map $\gamma : \pi(K) \to W$ such that for all $x\in \pi(K)$, $\gamma (x) \in K_x$ is called a \emph{section of $\pi$}. When there is no ambiguity on the map $\pi$ we will simply say that $\gamma$ is a \emph{section}.
\end{definition}

Using this notion we are now able to define our main object of study. In this paper \emph{measurable} is always intended with respect to the Borelians.

\begin{definition}\label{def:fiberbody}
    The \emph{fiber body} of $K$ with respect to the projection $\pi$ is the convex body
    \begin{equation}
        \Sp{K}:=\set{\int_{\pi(K)}\gamma(x) \dd x}{\gamma :\pi(K) \to W \hbox{ measurable section}} \in \KK(W).
    \end{equation}
    Here $\dd x$ denotes the integration with respect to the $n$--dimensional Lebesgue measure on $V$.  We say that a section $\gamma$ \emph{represents} $y\in \Sp{K}$ if $y=\int_{\pi(K)}\gamma(x) \dd x$. 
\end{definition}

\begin{remark}
    Note that, with this setting, if $\pi(K)$ is of dimension $<n$, then its fiber body is $\Sp{K}=\{0\}$. 
\end{remark}

This definition of fiber bodies, that can be found for example in \cite{esterovmixedfiber} under the name \emph{Minkowski integral}, extends the classic construction of fiber polytopes \cite{fiberpolytopes}, up to a constant. 
Here, we choose to omit the normalization $\frac{1}{\Vol(\pi(K))}$ in front of the integral used by Billera and Sturmfels in order to make apparent the \emph{degree} of the map $\Sp$ seen in \eqref{eq:degreeofSp}. This degree becomes clear with the notion of \emph{mixed fiber body}, see~\cite[Theorem~$1.2$]{esterovmixedfiber}.

\begin{proposition}\label{prop:n+1hom}
    For any $\lambda\in\R$ we have $\Sp{(\lambda K)} =\lambda |\lambda|^{n} \Sp{K}$. In particular if $\lambda\geq 0$
    \begin{equation}\label{eq:degreeofSp}
        \Sp{(\lambda K)} =\lambda^{n+1} \Sp{K}.
    \end{equation}
\end{proposition}
\begin{proof}
If $\lambda=0$ it is clear that the fiber body of $\{0\}$ is $\{0\}$.
Suppose now that $\lambda\neq 0$ and let $\gamma:\pi(K)\to W$ be a section. We can define another section $\tilde{\gamma}: \pi(\lambda K) \to W$ by $\tilde{\gamma}(x):=\lambda \gamma \left( \frac{x}{\lambda} \right)$. Using the change of variables $y=x/\lambda$, we get that
\begin{equation}
	\int_{\lambda \pi(K)}\tilde{\gamma}(x)\ \dd x = \lambda |\lambda|^{n} \int_{\pi(K)}\gamma(y)\  \dd y.
\end{equation}
This proves that $\Sp{\lambda K} \subseteq \lambda |\lambda|^{n} \Sp{K}$. Repeating the same argument for $\lambda^{-1}$ instead of $\lambda$, the other inclusion follows.
\end{proof}

\begin{corollary}
	If $K$ is centrally symmetric then so is $\Sp{K}$.
	\begin{proof}
	 Apply the previous proposition with $\lambda=-1$ to get $\Sigma_\pi\left((-1) K\right)=(-1)\Sigma_\pi K$. If $K$ is centrally symmetric with respect to the origin then $(-1)K=K$ and the result follows. The general case is obtained by a translation.
	\end{proof}
\end{corollary}

As a consequence of the definition, it is possible to deduce a formula for the support function of the fiber body. This is the rigorous version of equation~\eqref{eq:intro}.

\begin{proposition} \label{prop:supportaverage} 
For any $u\in W$ we have
\begin{equation}\label{eq:supportaverage}
    h_{\Sp{K}}(u)=\int_{\pi(K)} h_{K_x}(u) \dd x.
\end{equation}
\end{proposition}
\begin{proof}
By definition
\begin{equation}
    h_{\Sp{K}}(u)=\sup\set{\int_{\pi(K)} \langle u, \gamma(x) \rangle \ \dd x}{\gamma\hbox{ measurable section}} \leq \int_{\pi(K)} h_{K_x}(u) \dd x.
\end{equation}
To obtain the equality, it is enough to show that there exists a measurable section $\gamma_u:\pi(K)\to W$ with the following property: for all $x\in \pi(K)$ the point $\gamma_u(x)$ maximizes the linear form $\langle u, \cdot \rangle$ on $K_x$. In other words for all $x\in \pi(K)$, $\langle u, \gamma_u(x)\rangle=h_{K_x}(u)$. This is due to~\cite[Proposition~2.1]{AumannInt}.
\end{proof}

A similar result can be shown for the faces of the fiber body.

\begin{definition}
    Let $K\in \KK(\R^{n+m})$ and let $u\in \R^{n+m}$. We denote by $K^u$ the face of $K$ in direction $u$, that is all the points of $K$ that maximize the linear form $\langle u, \cdot \rangle$: 
    \begin{equation}
        K^u:=\set{y\in K}{\langle u, y\rangle=h_K(u)}.
    \end{equation}
    Moreover, if $\mathcal{U}=\{u_1,\ldots,u_k\}$ is an ordered family of vectors of $\R^{n+m}$, we write 
    \begin{equation}
        K^{\mathcal{U}}:=\left(\cdots\left(K^{u_1}\right)^{u_2}\cdots\right)^{u_k}.
    \end{equation}
\end{definition}

Note that $K^u$ is usually called an \emph{exposed} face of $K$. The notion of faces and exposed faces coincide for polytopes but are different in general. In this paper we only consider exposed faces that we call faces for simplicity. In the following, we show that the face of the fiber body is, in some sense, the fiber body of the faces. 

\begin{lemma}\label{lem:sectionrepdirection}
     Let $\mathcal{U}=\{u_1,\ldots,u_k\}$ be a an ordered family of linearly independent vectors of $W$, take $y\in\Sp{K} $ and let $\gamma: \pi(K)\to W$ be a section that represents $y$. Then $y\in\left(\Sp{K}\right)^\mathcal{U}$ if and only if $\gamma (x)\in \left(K_x\right)^\mathcal{U}$ for almost all $x\in \pi(K)$. In particular we have that 
     \begin{equation}\label{eq:repofface}
        \left( \Sp{K}\right)^{\mathcal{U}}=\set{\int_{\pi(K)}\gamma(x) \dd x}{\gamma \text{ section such that} \ \gamma(x)\in \left(K_x\right)^{\mathcal{U}} \text{ for all }x}.
     \end{equation}

    \begin{proof}
        Suppose first that $\mathcal{U}=\{u\}.$
        Assume that $\gamma(x)$ is not in $\left(K_x\right)^u$ for all $x$ in a set of non--zero measure $\mathscr{O}\subset \pi(K)$. Then there exists a measurable function $\xi:\pi(K)\to W$ with $\langle u,\xi\rangle\geq0$ and $\langle u,\xi(x)\rangle>0$ for all $x\in\mathscr{O}$, such that $\tilde{\gamma}:=\gamma+\xi$ is a section (for example you can take $\tilde{\gamma}(x)$ to be the nearest point on $K_x$ of $\gamma(x)+u$). Let $\tilde{y}:=\int_{\pi(K)}\tilde{\gamma}$ . Then $\langle u,\tilde{y}\rangle =\langle u, y\rangle + \int_{\pi(K)}\langle u, \xi\rangle >\langle u, y\rangle$. Thus $y$ does not belong to the face $\left(\Sp{K}\right)^u$.
        
        Suppose now that $y$ is not in the face $\left(\Sp{K}\right)^u$. Then there exists $\tilde{y}\in \Sp{K}$ such that $\langle u, \tilde{y}\rangle > \langle u, y\rangle$. Let $\tilde{\gamma}$ be a section that represents $\tilde{y}$. It follows that $\int_{\pi(K)}\langle u, \tilde{\gamma}\rangle > \int_{\pi(K)}\langle u, \gamma \rangle$. This implies the existence of a set $\mathscr{O}\subset \pi(K)$ of non--zero measure where $\langle u, \tilde{\gamma}(x)\rangle > \langle u, \gamma(x) \rangle$ for all $x\in\mathscr{O}$. Thus for all $x\in\mathscr{O}$, $\gamma(x)$ does not belong to the face $\left(K_x\right)^u$.
        
        In the case $\mathcal{U}=\{u_1,\ldots,u_{k+1}\}$ we can apply inductively the same argument. Replace $\Sp{K}$ by $\left(\Sp{K}\right)^{\{u_1,\ldots,u_k\}}$ and $u$ by $u_{k+1}$, and use the representation of $\left(\Sp{K}\right)^{\{u_1,\ldots,u_k\}}$ given by~\eqref{eq:repofface}.
    \end{proof}
\end{lemma}

Using the same strategy in the proof of Proposition~\ref{prop:supportaverage} we obtain the following formula.

\begin{lemma}\label{lem:segment}
    For every $u,v\in W$, $h_{ (\Sp{K})^u}(v)=\int_{\pi(K)} h_{(K_x)^u}(v)\ \dd x$.
\end{lemma}

The fiber body behaves well under the action of $\hbox{GL}(V)\oplus\hbox{GL}(W)$ as a subgroup of $\hbox{GL}(\R^{n+m}).$

\begin{proposition}
    Let $g_n\in\hbox{GL}(V)$, $g_m\in \hbox{GL}(W)$ and $K\in \KK(\R^{n+m})$. Then
    \begin{equation}
        \Sp{\Big( (g_n \oplus g_m)(K) \Big)} = |\det(g_n)| \cdot g_m \Big( \Sp{K} \Big).
    \end{equation}
\end{proposition}
\begin{proof}
    This is a quite straightforward consequence of the definitions. After observing that
    \begin{equation}
        \Big( (g_n \oplus g_m)(K) \Big)_x = g_m \Big( K_{g_n^{-1}(x)} \Big)
    \end{equation}
    and $\pi\left(  (g_n \oplus g_m)(K)\right)= g_n \pi(K)$, use equation~\eqref{eq:supportaverage} with the change of variables $x\mapsto g_n^{-1}x$. By Proposition~\ref{prop:propertieshK}--$\mathit{(ii)}$ 
    we have $h_{g_m K_x}(u)=h_{K_x}(g_m^T u)$, so the thesis follows.
\end{proof}

\subsection{Regularity of the sections}
By definition, a point $y$ of the fiber body $\Sp{K}$ is the integral $y=\int_{\pi(K)}\gamma(x) \dd x$ of a \emph{measurable} section $\gamma$. Thus $\gamma$ can be modified on a set of measure zero without changing the point $y$, i.e. $y$ only depends on the $L^1$ class of $\gamma$. It is natural to ask what our favourite representative in this $L^1$ class will be and how regular can it be. In the case where $K$ is a polytope, $\gamma$ can always be chosen continuous. However if $K$ is not a polytope and if $y$ belongs to the boundary of $\Sp{K}$, a continuous representative may not exist. This is due to the fact that, in general, the map $x\mapsto K_x$ is only upper semicontinuous, see~\cite[Section~$6$]{KhovanskiiFamily}. 

\begin{example}\label{ex:counterex_continuity}
 Consider the function $f : S^1 \to \R$ such that
\begin{equation}
    f(x,y) =    \begin{cases}
                    0 & x<0 \\
                    1 & x\geq 0
                \end{cases}
\end{equation}
and let $K := \conv(\text{graph}(f)) \subset \R^3$ in Figure~\ref{fig:counterex_continuity}. This is a semialgebraic convex body, whose boundary may be subdivided in $8$ distinct pieces: two half--discs lying on the planes $\{z=0\}$ and $\{z=1\}$, two triangles with vertices $(-1,0,0),(0,\pm 1,1)$ and $(1,0,1),(0,\pm 1,0)$ respectively, four cones with vertices $(0,\pm 1,0),(0,\pm 1, 1)$.
\begin{figure}
    \centering
    \includegraphics[width=0.4\textwidth]{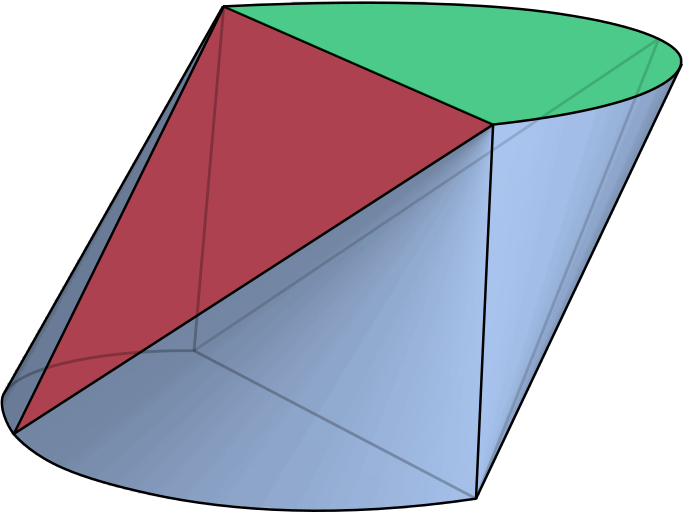}
    \caption{The convex body of Example~\ref{ex:counterex_continuity}. In its boundary there are $2$ green half--discs, $2$ red triangles and $4$ blue cones.}
    \label{fig:counterex_continuity}
\end{figure}
Let $\pi : \R^3\to \R$ be the projection on the first coordinate $\pi(x,y,z)=x$. Then the point $p\in\Sp{K}\subset \R^2$ maximizing the linear form associated to $(y,z)=(1,0)$ must have only non--continuous sections. This can be proved using the representation of a face given by~\eqref{eq:repofface}.
\end{example}

We prove that most of the points of the fiber body have a continuous representative.

\begin{proposition}
    Let $K\in \KK(\R^{n+m})$ and let $\Sp{K}$ be its fiber body. The set of its points that can be represented by a continuous section is convex and dense. In particular, all interior points of $\Sp{K}$ can be represented by a continuous section.
\end{proposition}
\begin{proof}
    Consider the set
    \begin{equation}
        C = \set{\int_{\pi(K)}\gamma(x) \dd x}{\gamma :\pi(K) \to K \text{ continuous section}}
    \end{equation}
    that is clearly contained in the fiber body $\Sp{K}$. It is convex: take $a,b \in C$ represented by continuous sections $\alpha,\beta :\pi(K) \to K$ respectively. Then any convex combination can be written as $c = t a + (1-t) b = \int_{\pi(K)} \Big( t \alpha(x) + (1-t)\beta(x)\Big) dx$. Since $t \alpha + (1-t)\beta$ is a continuous section for any $t\in[0,1]$, $C$ is convex.
    
    We now need to prove that the set $C$ is also dense in $\Sp{K}$. Let $\gamma$ be a measurable section; by definition it is a measurable function $\gamma:\pi(K) \to W$, such that $\gamma(x)\in K_x$ for all $x\in\pi(K)$. For every $\epsilon>0$ there exists a continuous function $g:\pi(K) \to W$ with $\| \gamma - g\|_{L^1}<\epsilon$, but this is not necessarily a section of $K$, since a priori $g(x)$ can be outside $K_x$. Hence define $\tilde{\gamma}:\pi(K) \to W$ such that
    \begin{equation}
        \tilde{\gamma}(x) = p\Big( K_x, g(x) \Big)
    \end{equation}
    where $p(A,a)$ is the nearest point map at $a$ with respect to the convex set $A$. By \cite[Lemma~$1.8.11$]{bible} $\tilde{\gamma}$ is continuous and by definition $\text{graph}(\tilde{\gamma}) \subset K$. Therefore $\int_{\pi(K)} \tilde{\gamma} \in C$. Moreover,
    \begin{equation}
        \| \gamma - \tilde{\gamma}\|_{L^1} \leq \| \gamma - g\|_{L^1} <\epsilon
    \end{equation}
    hence the density is proved. As a consequence we get that $ \text{int}\Sp{K} \subseteq C \subseteq \Sp{K}$ so all the interior points of the fiber body have a continuous representative.
\end{proof}

To our knowledge, the regularity of the sections needed to represent all points is not known. 

\subsection{Strict convexity} 
In the case where $K^u$ consists of only one point we say that $K$ is \emph{strictly convex in direction $u$}. Moreover, a convex body is said to be \emph{strictly convex} if it is strictly convex in every direction. We now investigate this property for fiber bodies.

\begin{proposition}\label{prop:strictlyconvex}
    Let $K\in\KK(\R^{n+m})$ and let us fix a vector $u\in W$. The following are equivalent:
    \begin{enumerate}
        \item $\Sp{K}$ is strictly convex in direction $u$;  \\
        \item almost all the fibers $K_x$ are strictly convex in direction $u$.
    \end{enumerate}
\end{proposition}

\begin{proof}
By Proposition~\ref{prop:propertieshK}--$\mathit{(iii)}$, a convex body is strictly convex in direction $u$ if and only if its support function is $\mathcal{C}^1$ at $u$. Therefore, if almost all the fibers $K_x$ are strictly convex in $u$, then, the convex body being compact, the support function $h_{\Sp{K}}(u)=\int_{\pi(K)} h_{K_x}(u) \dd x$ is $\mathcal{C}^1$ at $u$, i.e. the fiber body is strictly convex in that direction. 

Now suppose that $\Sp{K}$ is strictly convex in direction $u$, i.e. $\left(\Sp{K}\right)^u$ consists of just one point $y$. This means that the support function of this face is linear and it is given by $\langle y, \cdot \rangle$. We now prove that the support function of $K_x^u$ is linear for almost all $x$, and this will conclude the proof. Lemma~\ref{lem:segment} implies that
\begin{equation}
    h_{\left(\Sp{K}\right)^u} = \int_{\pi(K)}h_{K_x^u} \dd x= \langle y, \cdot \rangle.
\end{equation}
For any two vectors $v_1,v_2$, we have
\begin{equation}
    \langle y, v_1 + v_2 \rangle = \int_{\pi(K)} h_{K_x^u}(v_1+v_2) dx \leq \int_{\pi(K)} h_{K_x^u}(v_1) dx + \int_{\pi(K)} h_{K_x^u}(v_2) dx = \langle y, v_1 \rangle + \langle y, v_2 \rangle
\end{equation}
thus the inequality in the middle must be an equality. But since $h_{K_x^u}(v_1+v_2) \leq  h_{K_x^u}(v_1) + h_{K_x^u}(v_2)$, we get that this is an equality for almost all $x$, i.e. the support function of $K_x^u$ is linear for almost every $x\in {\pi(K)}$. Therefore almost all the fibers are strictly convex.
\end{proof}

The elliptope in Section \ref{sec:elliptope} furnishes an example of a convex body $\mathcal{E}$ and a projection $\pi$ such that the fiber body $\Sp{\mathcal{E}}$ is strictly convex, but the two fibers $\mathcal{E}_{\pm 1}$ are segments, hence not strictly convex.

\section{Puffed polytopes}\label{sec:Puffed}

In this section we introduce a particular class of convex bodies arising from polytopes. A known concept in the context of hyperbolic polynomials and hyperbolicity cones is that of the \emph{derivative cone}; see \cite{renegarhyperbolic} or \cite{sanyalderivativecones}. Since we are dealing with compact objects, we will repeat the same construction in affine coordinates, i.e., for polytopes instead of polyhedral cones.

Let $P$ be a full--dimensional polytope in $\R^N$, containing the origin, with $d$ facets given by affine equations $l_1(x_1,\ldots ,x_N)=a_1, \ldots ,l_d(x_1,\ldots,x_N)=a_d$. Consider the polynomial
\begin{equation}\label{eq:polytope}
    p(x_1, \ldots ,x_N) = \prod_{i=1}^d \left(l_i(x_1,\ldots ,x_N) -a_i\right).
\end{equation}
Its zero locus is the algebraic boundary of $P$, i.e. the algebraic closure of the boundary, in the Zariski topology, as in \cite{sinn_alg_bound}. Consider the homogenization of $p$, that is $\tilde{p}(x_1, \ldots ,x_N,w)= \prod_{i=1}^d \left(l_i(x_1,\ldots ,x_N) -a_i w\right)$. It is the algebraic boundary of a polyhedral cone and it is hyperbolic with respect to the direction $(0,\ldots,0,1)\in \R^{N+1}$. Then for all $i<d$ the polynomial
\begin{equation}\label{eq:der_puff}
    \left( \frac{\partial^i}{\partial w^i} \tilde{p} \right) (x_1,\ldots ,x_N, 1)
\end{equation}
is the algebraic boundary of a convex set containing the origin, see \cite{sanyalderivativecones}. This allows us to introduce the following definition.

\begin{definition}\label{def:puffed}
    Let $Z_i$ be the zero locus of~\eqref{eq:der_puff} in $\R^N$. The \emph{$i$-th puffed $P$} is the closure of the connected component of the origin in $\R^N\setminus Z_i$. We denote it by $\puff{i}{P}$.
\end{definition}
In particular the puffed polytopes are always spectrahedra \cite[Corollary 1.3]{Branden:HyperbolicityCones}.
As the name suggests, the puffed polytopes $\puff{i}{P}$ are fat, inflated versions of the polytope $P$ and in fact contain $P$. On the other hand, despite the definition involves a derivation, the operation of ``taking the puffed'' does not behave as a derivative. In particular, it does not commute with the Minkowski sum, that is, in general for polytopes $P_1, P_2$:
\begin{equation}
    \puff{1}{P_1+P_2}\neq \puff{1}{P_1} + \puff{1}{P_2}.
\end{equation}
To show this with, we build a counterexample in dimension $N=2$.
\begin{example}
    Let us consider two squares $P_1 = \conv\{(\pm 1, \pm 1)\}$, $P_2 = \conv\{(0, \pm 1), (\pm 1, 0)\} \subset \R^2$.
    The first puffed square is a disc with radius half of the diagonal, so $\puff{1}{P_1}$ has radius $\sqrt{2}$ and $\puff{1}{P_2}$ has radius $1$. Therefore $\puff{1}{P_1} + \puff{1}{P_2}$ is a disc centered at the origin of radius $1+\sqrt{2}$. On the other hand $P_1 + P_2$ is an octagon. Its associated polynomial in \eqref{eq:polytope} is
    \begin{equation}
        p(x,y) = ((x+y)^2-9)((x-y)^2-9)(x^2-4)(y^2-4).
    \end{equation}
    Via the procedure explained above we obtain the boundary of this puffed octagon, as the zero locus of the following irreducible polynomial
    \begin{equation}
        2 x^6+7 x^4 y^2+7 x^2 y^4+2 y^6-88 x^4-193 x^2 y^2-88 y^4+918 x^2+918 y^2-2592.
    \end{equation}
    This is a curve with three real connected components, shown in violet in Figure~\ref{fig:puff_octa}.
    Clearly the puffed octagon is not a circle, hence $\puff{1}{P_1} + \puff{1}{P_2} \neq \puff{1}{P_1+P_2}$.
    \begin{figure}
    \centering
    \includegraphics[width=0.5\textwidth]{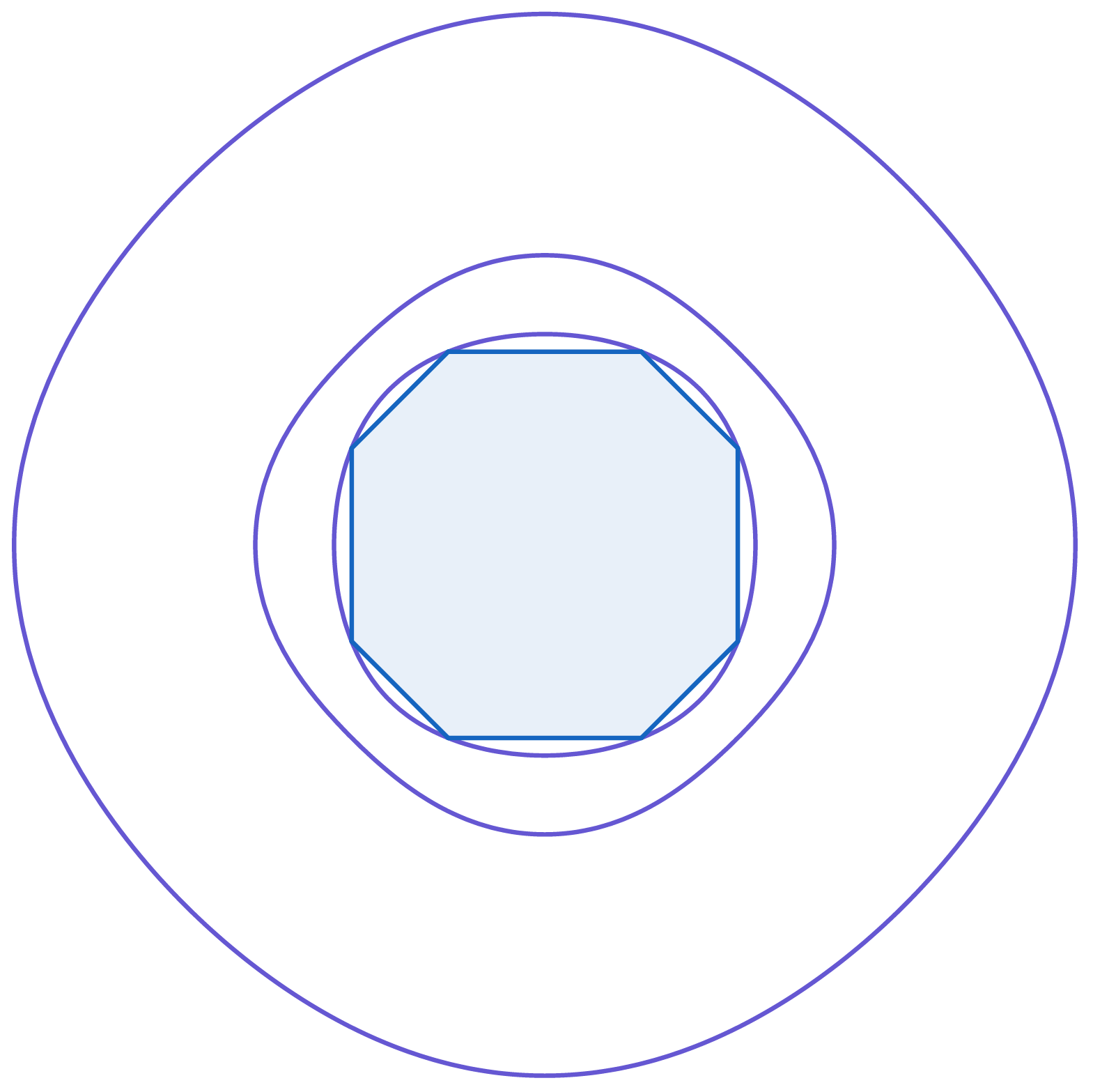}
    \caption{The octagon, in blue, and (the algebraic boundary of) its puffed octagon, in violet.}
    \label{fig:puff_octa}
    \end{figure}
\end{example}

\subsection{Strict convexity of the puffed polytopes}

Our aim is to study the strict convexity of the fiber body of a puffed polytope. In order to do so, we shall at first say something more about the boundary structure of a puffed polytope itself. In particular, we will see that the appropriate quantity to consider is the \emph{multiciplicity} of the faces, that is, their multiciplicity as zeroes of the polynomial defining the algebraic boundary. Indeed, a face $F\subset P$ will be in the boundary of $\puff{i}{P}$ for all $i$ less or equal than the multiplicity of $F$.

\begin{lemma}\label{lem:faces_puffedP}
    Let $P\subset \R^N$ be a full--dimensional polytope. Then all faces $F$ of $P$ of dimension $k<N-i$, are contained in the boundary of $\puff{i}{P}$.
\end{lemma}
\begin{proof}
    Let $F$ be a $k-$face of $P$; it is contained in the zero set of the polynomial \eqref{eq:polytope}. Moreover $F$ arises as the intersection of at least $N-k$ facets (i.e. faces of dimension $N-1$), thus its points are zeros of multiplicity at least $N-k$. Hence, if $N-k>i$ the face $F$ is still in the zero set of \eqref{eq:der_puff}, i.e. it belongs to the boundary of $\puff{i}{P}$.
\end{proof}

The other direction is not always true: there may be $k$--faces of $P$, with $k\geq N-i$, whose points are zeros of \eqref{eq:der_puff} of multiplicity higher than $i$, and hence faces of $\puff{i}{P}$. However there are two cases in which this is not possible.
\begin{lemma}\label{lem:facesofpuff12}
Let $P\subset \R^N$ be a full--dimensional polytope.
\begin{itemize}
    \item{$i=1$:} the flat faces in the boundary of $\puff{1}{P}$ are exactly the faces of dimension $k<N-1$; \\
    \item{$i=2$:} the flat faces in the boundary of $\puff{2}{P}$ are exactly the faces of dimension $k<N-2.$
\end{itemize}
    \begin{proof}
     The first point is clear because the facets (faces of dimension $N-1$) are the only zeroes of multiplicity one. The second point follows from the so called ``diamond property'' of polytopes \cite{ziegler2012lectures}.
    \end{proof}
\end{lemma}

\begin{remark}
    By \cite[Proposition~$24$]{renegarhyperbolic} we can deduce that the flat faces of a puffed polytope must be faces of the polytope itself. The remaining points in the boundary of $\puff{i}{P}$ are exposed points.
\end{remark}

Using this result we can deduce conditions for the strict convexity of the fiber body of a puffed polytope.
\begin{proposition}[Fiber $1$st puffed polytope]\label{prop:puff1}
    Let $P\subset \R^{n+m}$ be a full--dimensional polytope, $n\geq 1$, $m\geq 2$, and take any projection $\pi : \R^{n+m} \to \R^n$. The fiber puffed polytope $\Sigma_{\pi}\left( \puff{1}{P} \right)$ is strictly convex if and only if $m=2$.
    \begin{proof}
     By Lemma~\ref{lem:facesofpuff12}, the flat faces in the boundary of $\puff{1}{P}$ are the faces of $P$ of dimension $k<n+m-1$. Suppose first that $m>2$ and let $F$ be a $(n+m-2)$--face of $P$. Take a point $p$ in the relative interior of $F$ and let $x_p:=\pi(p)$. Then the dimension of $F\cap \pi^{-1}(x_p)$ is at least $m-2\geq 1$; we can also assume without loss of generality
     that
\begin{equation}\label{eq:dim}
    1\leq \dim \left( F\cap \pi^{-1}(x_p) \right) < n+m-2.
\end{equation}
Furthermore there is a whole neighborhood $U$ of $x_p$ such that condition \eqref{eq:dim} holds, so for every $x\in U$ the convex body $\left( \puff{1}{P} \right)_x$ is not strictly convex. By Proposition~\ref{prop:strictlyconvex} then $\Sp{\left( \puff{1}{P}\right) }$ is not strictly convex.
Suppose now that $m=2$ and fix a flat face $F$ of $\puff{1}{P}$. Its dimension is less or equal than $n$, so $\left( F\cap \pi^{-1}(x_p) \right)$ is either one point or a face of positive dimension. In the latter case $\dim \pi(F)\leq n-1$, i.e. it is a set of measure zero in $\pi \left( \puff{1}{P} \right)$. Because there are only finitely many flat faces, we can conclude that almost all the fibers are strictly convex and thus by Proposition~\ref{prop:strictlyconvex}, $\Sp{\left( \puff{1}{P}\right) }$ is strictly convex.
    \end{proof}
\end{proposition}

A similar result holds for the second fiber puffed polytope, using Lemma~\ref{lem:facesofpuff12}.

\begin{proposition}[Fiber $2$nd puffed polytope]\label{prop:puff2}
    Let $P\subset \R^{n+m}$ be a full--dimensional polytope, $n\geq 1$, $m\geq 2$, and take any projection $\pi : \R^{n+m} \to \R^n$. The fiber puffed polytope $\Sigma_{\pi}\left( \puff{2}{P} \right)$ is strictly convex if and only if $m\leq 3$, i.e. $m=2$ or $3$.
\end{proposition}
\begin{proof}
    We can use the previous strategy again. If $m>3$, there always exists a face of $\puff{2}{P}$ of dimension $n+m-3$ whose non--empty intersection with fibers of $\pi$ has dimension at least $1$ and strictly less than $n+m-3$. So in this case we get a non strictly convex fiber body. On the other hand, when $m=2$ or $3$ the intersection of the fibers and the flat faces has positive dimension only on a measure zero subset of $\R^n$, hence almost all the fibers are strictly convex and the thesis follows.
\end{proof}

Can we generalize this result for the $i$-th puffed polytope? In general no, and the reason is precisely that a $k$-face may be contained in more than $(n+m-k)$ facets, when $k<n+m-2$. 
The polytopes $P$ for which this does not happen are called \emph{simple polytopes}. Thus with the same proof as above we obtain the following.
\begin{proposition}[Fiber $i$-th puffed simple polytope]\label{prop:puffn}
    Let $P\subset \R^{n+m}$ be a full--dimensional simple polytope, $n\geq 1$, $m\geq 2$, and take any projection $\pi : \R^{n+m} \to \R^n$. The fiber puffed polytope $\Sigma_{\pi}\left( \puff{i}{P} \right)$ is strictly convex if and only if $m\leq i+1$.
\end{proposition}
In the case where $P$ is not simple, one has to take into account the number of facets in which each face of dimension $k \geq n+m-i$ is contained, in order to understand if they are or not part of the boundary of $\puff{i}{P}$.

\subsection{A case study: the elliptope}\label{sec:elliptope}
Take the tetrahedron $\mathcal{T}$ in $\R^3$ realized as
\begin{equation}
    \conv\{ (1,1,1), (1,-1,-1), (-1,1,-1), (-1,-1,1) \}.
\end{equation}
The first puffed tetrahedron (for the rest of the subsection we will omit the word ``first'') is the semialgebraic convex body called the \emph{elliptope} which is the set of points $(x,y,z)\in[-1,1]^3$ such that $x^2+y^2+z^2-2xyz\leq 1$.
Let $\pi$ be the projection on the first coordinate: $\pi(x,y,z)=x$. The fibers of the elliptope at $x$ for $x\in(-1,1)$ are the ellipses defined by 
\begin{equation}
    \mathcal{E}_x=\set{(y,z)}{\left(\frac{y-xz}{\sqrt{1-x^2}}\right)^2+z^2\leq 1}.
\end{equation}
Introducing the matrix 
\begin{equation}
    M_x:=
    \begin{pmatrix}   
    \frac{1}{\sqrt{1-x^2}}  & \frac{-x}{\sqrt{1-x^2}} \\
    0 & 1  \end{pmatrix}
\end{equation}
it turns out that $\mathcal{E}_x = \set{(y,z)}{\|M_x(y,z)\|^2\leq 1} = (M_x)^{-1}B^2$, where $B^2$ is the unit $2$--disc. We obtain 
\begin{equation}
    h_{\mathcal{E}_x}(u,v)=h_{B^2}\left((M_x)^{-T}(u,v)\right)=\left\|(M_x)^{-T}(u,v)\right\| = \sqrt{u^2+v^2+2 x u v}.
\end{equation}
By~\eqref{eq:supportaverage} we need to compute the integral of $h_{\mathcal{E}_x}$ between $x=-1$ and $x=1$ to obtain the support function of the fiber body of the elliptope. We get
\begin{equation}
    h_{\Sigma_{\pi}\mathcal{E}}(u,v)= \frac{1}{3 u v}\left( |u+v|^3-|u-v|^3 \right).
\end{equation}
Hence the fiber body is semialgebraic and its algebraic boundary is the zero set of the four parabolas $3y^2 + 8z - 16$, $3y^2 - 8z - 16$, $8y + 3z^2 - 16$, $8y - 3z^2 + 16$, displayed in Figure~\ref{fig:fiber_ell}.
\begin{figure}
    \begin{subfigure}{0.6\textwidth}
        \centering
        \includegraphics[width = 0.9\textwidth]{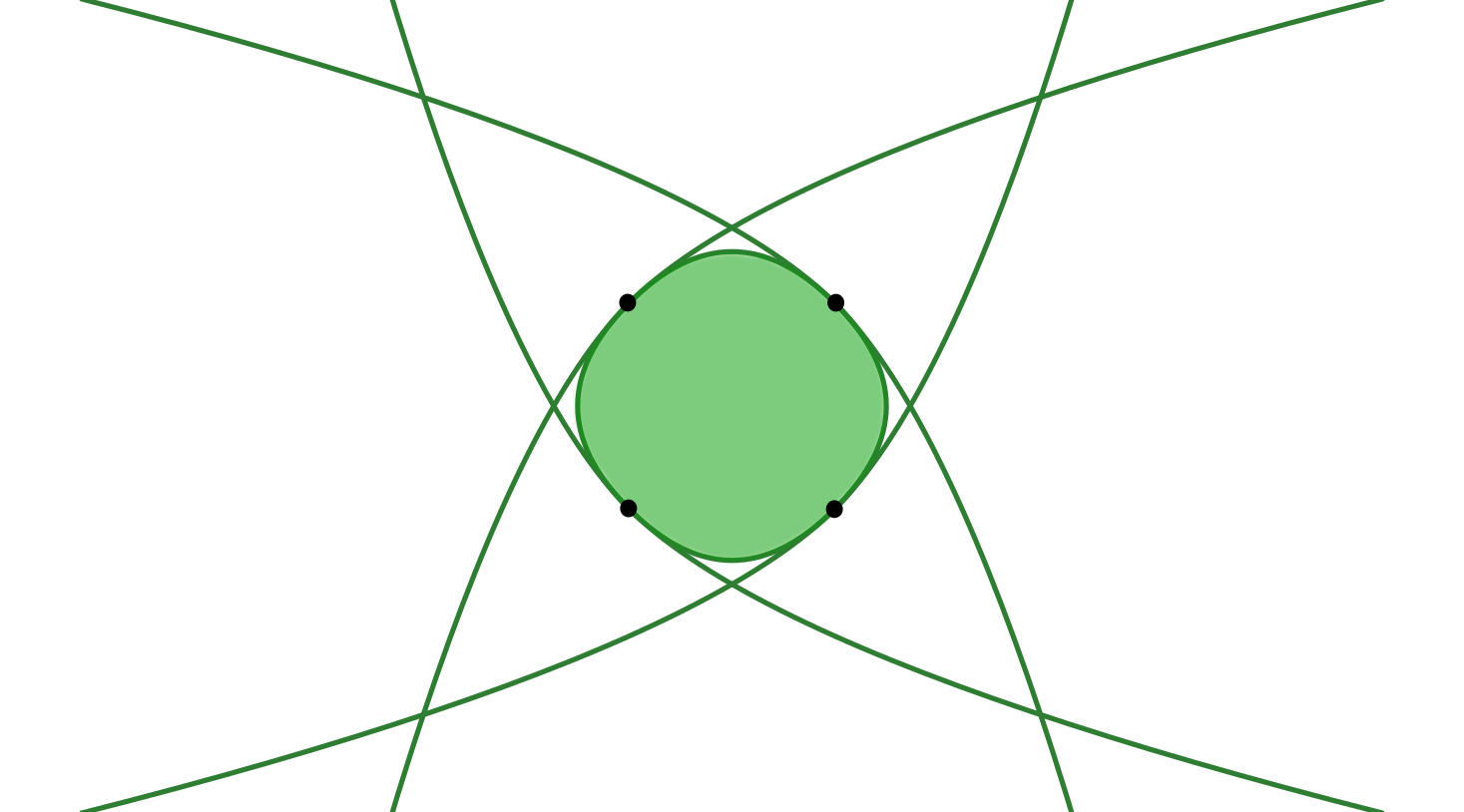}
        \caption{}
        \label{fig:fiber_ell}
    \end{subfigure}
    \begin{subfigure}{0.39\textwidth}
        \centering
        \includegraphics[width=0.6\textwidth]{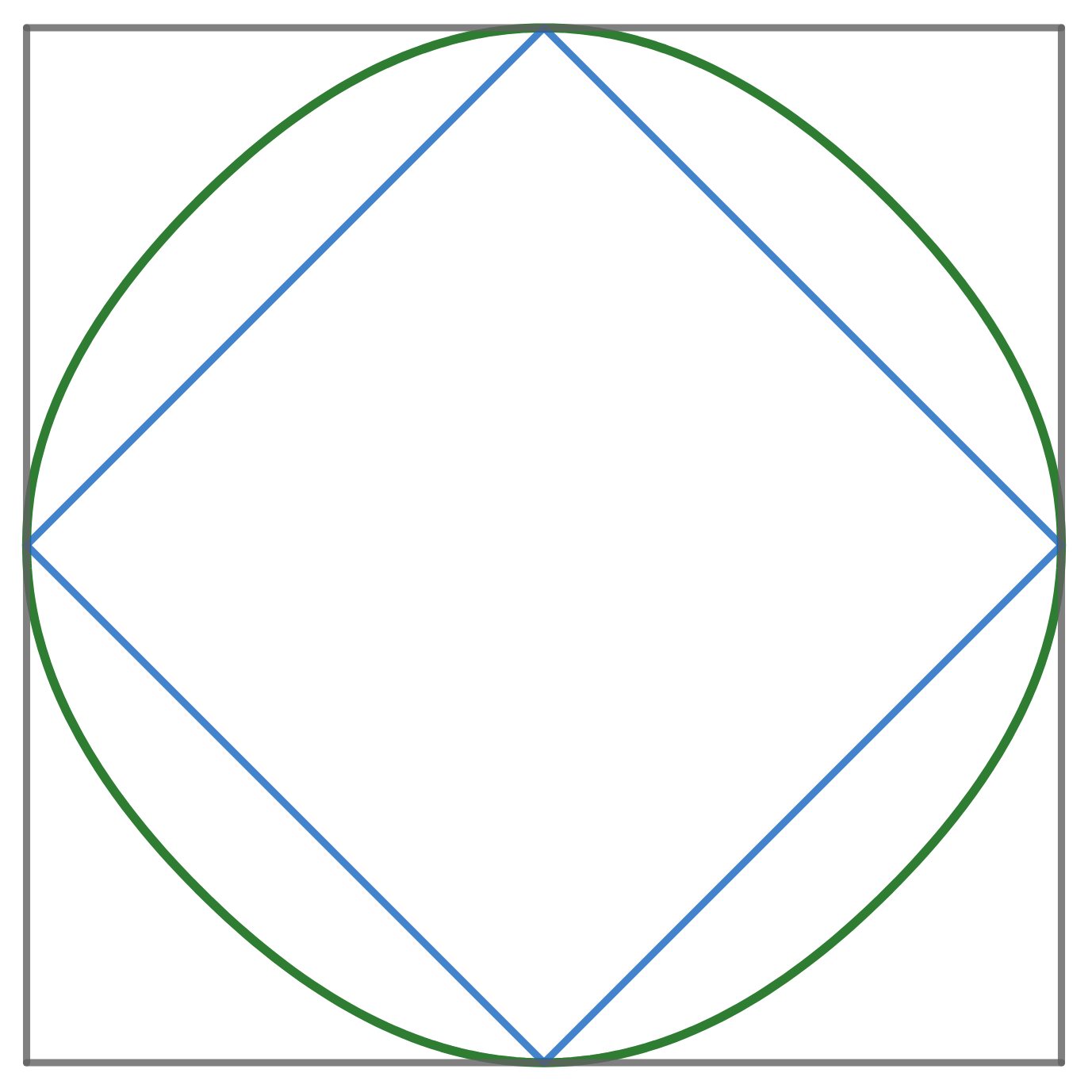}
        \caption{}
        \label{fig:3_fiberbodies}
    \end{subfigure}
    \caption{Left: the four green parabolas meet in the four black points on the boundary of the fiber elliptope, that lie on the diagonals $y=z$ and $y=-z$ .\\ Right: sandwiched fiber bodies. The blue square is the fiber tetrahedron $\Sigma_{\pi}\mathcal{T}$; the green convex body is the fiber elliptope $\Sigma_{\pi}\mathcal{E}$; the grey square is the fiber cube $\Sigma_{\pi}\left( [-1,1]^3\right)$.}
\end{figure}

As anticipated in Proposition~\ref{prop:puff1} the fiber elliptope is strictly convex.
Notice that the elliptope is naturally sandwiched between two polytopes: the tetrahedron $\mathcal{T}$ and the cube $[-1,1]^3$.
Therefore, as a natural consequence of the definition, the same chain of inclusions works also for their fiber bodies:
\begin{equation}
    \Sigma_{\pi}\mathcal{T} \subset \Sigma_{\pi}\mathcal{E} \subset \Sigma_{\pi}\left([-1,1]^3\right)
\end{equation}
as shown in Figure~\ref{fig:3_fiberbodies}.

\begin{remark}
    From this example it is clear that the operation of ``taking the fiber body'' does not commute with the operation of ``taking the puffed polytope''. In fact the puffed polytope of the blue square in Figure~\ref{fig:3_fiberbodies} is not the green convex body bounded by the four parabolas: it is the disc $y^2 + z^2 \leq 4$.
\end{remark}

\section{Curved convex bodies}\label{sec:Smooth} 
In this section we are interested in the case where the boundary of the convex body $K$ is highly regular. We prove Theorem~\ref{thm:supofsmooth} which is a formula to compute support function of the fiber body directly in terms of the support function of $K$, without having to compute those of the fibers.
\begin{definition}
    We say that a convex body $K$ is \emph{curved} if the following two conditions are satisfied: the support function $h_K$ is $C^2$ and the gradient $\nabla h_K$ restricted to the sphere is a $C^1$ diffeomorphism with the boundary of $K$.
\end{definition}

In that case $K$ is full--dimensional and its boundary is a $C^2$ hypersurface. Moreover we have the following.
\begin{lemma}
    Let $K\subset \R^{n+m}$ be a curved convex body and let $v\in S^{n+m-1}$. Then the differential $\dd_v\nabla h_K$ is a symmetric positive definite automorphism of $v^\perp$.
    \begin{proof}
     This is proved in~\cite[p.$116$]{bible}, where curved convex bodies are said to be ``of class $C^2_+$'' and $\dd_v\nabla h_K$ is denoted by $\overline{W}_v$. 
    \end{proof}
\end{lemma}

 The following gives an expression for the face of the fiber body. This is to be compared with the case of polytopes which is given in~\cite[Lemma~$11$]{esterovkhovanskii}.

\begin{lemma}\label{lemma:supofsmooth}
If $K$ is a curved convex body and $u\in W$ with $\| u \|=1$, then 
\begin{equation}
	\nabla h_{\Sigma_\pi K}(u)=\int_V \nabla h_K(u+\xi)\cdot J_{\psi_u}(\xi)\ \dd\xi
\end{equation}
where $\psi_u:V\to V$ is given by $\psi_u(\xi)=\left(\pi\circ\nabla h_K\right)(u+\xi)$ and $J_{\psi_u}(\xi)$ denotes its Jacobian (i.e. the determinant of its differential) at the point $\xi$.
\begin{proof}
     From~\eqref{eq:repofface} we have that $\nabla h_{\Sp{K}}(u)=\int_{\pi(K)} \gamma_u(x) \dd x$, where $\gamma_u(x)=\nabla h_{K_x}(u).$ Assume that $x=\psi_u(\xi)$ is a change of variables. We get $\gamma_u(x)=(\gamma_u\circ\pi\circ\nabla h_K) (u+\xi)=\nabla h_K(u+\xi)$ and the result follows. 
     
     It remains to prove that it is indeed a change of variables. Note that $\nabla h_K(u+\xi)=\nabla h_K (v)$ where $v=\tfrac{u+\xi}{\|u+\xi\|}\in S^{n+m-1}$. The differential of the map $\xi\mapsto v$ maps $V$ to $\left(V+\R u\right)\cap v^\perp$. Moreover $\nabla h_K$ restricted to the sphere is a $C^1$ diffeomorphism by assumption. Thus it only remains to prove that its differential $\dd_v\nabla h_K$ sends $\left(V+\R u\right)\cap v^\perp$ to a subspace that does not intersect $\ker \left(\restr{\pi}{v^\perp}\right)$. To see this, note that $\ker \left(\restr{\pi}{v^\perp}\right)^\perp=\left(V+\R u\right)\cap v^\perp$. Moreover, by the previous lemma, we have that $\langle w, \dd_v\nabla h_K \cdot w\rangle =0$ if and only if $w=0$. Thus if $w\in\ker \left(\restr{\pi}{v^\perp}\right)^\perp$ and $w\neq 0$, then $\pi \left(\dd_v\nabla h_K \cdot w\right)\neq 0$. Putting everything together, this proves that $\dd_\xi \psi_u$ has no kernel which is what we wanted.
\end{proof}
\end{lemma}

As a direct consequence we derive a formula for the support function. 
\begin{theorem}\label{thm:supofsmooth}
     Let $K\subset \R^{n+m}$ be a curved convex body. Then the support function of $\Sp{K}$ is for all $u\in W$
    \begin{equation}\label{eq:supofsmooth}
         h_{\Sp{K}}(u)=\int_V \langle u, \nabla h_K(u+\xi)\rangle \cdot J_{\psi_u}(\xi)\ \dd\xi
    \end{equation}
   where $\psi_u:V\to V$ is given by $\psi_u(\xi)=\left(\pi\circ\nabla h_K\right)(u+\xi)$ and $J_{\psi_u}(\xi)$ denotes its Jacobian at the point $\xi$.
\end{theorem}
\begin{proof}
    Apply the previous lemma to $h_{\Sp{K}}(u)=\langle u, \nabla h_{\Sp{K}}(u)\rangle$.
\end{proof}

Assume that the support function $h_K$ is \emph{algebraic}, i.e. it is a root of some polynomial equation. Then, the integrand in Lemma~\ref{lemma:supofsmooth} and in Theorem~\ref{thm:supofsmooth} is also algebraic. Indeed, it is simply $ \nabla h_K(u+\xi)$ times the Jacobian of $\psi_u$ which is a composition of algebraic functions. 
We can generalize this concept in the direction of \emph{$D$--modules} (see \cite{Zeilberger:HolonomicSystems}, or \cite{Holonomic} for a text with a view towards applied nonlinear algebra). One can define what it means for a $D$--ideal of the Weyl algebra $D$ to be \emph{holonomic}. Then a function is holonomic if its annihilator, a $D$--ideal, is holonomic. Intuitively, this means that such function satisfies a system of linear homogeneous differential equations with polynomial coefficients, plus a suitable dimension condition. Holonomicity can be seen as a generalization of algebraicity which is closed under integration. 
We say that a convex body $K$ is \emph{holonomic} if its support function $h_K$ is holonomic. In this setting, the fiber body satisfies the following property.

\begin{corollary}\label{cor:holonomiccurved}
    If K is a curved holonomic convex body, then its fiber body is again holonomic.
    \begin{proof}
        We prove that the integrand in Theorem~\ref{thm:supofsmooth} is a holonomic function of $u$ and $\xi$. Then the result follows from the fact that the integral of a holonomic function is holonomic \cite[Proposition~$2.11$]{Holonomic}. If $h_K$ is holonomic then $\nabla h_K (u+\xi)$ is a holonomic function of $u$ and $\xi$, as well as its scalar product with $u$. It remains to prove that the Jacobian of $\psi_u$ is holonomic. 
        But $\psi_u$ is the projection of a holonomic function and thus holonomic, so the result follows.
    \end{proof}
\end{corollary}

\subsection{A case study: Schneider's polynomial body}

In~\cite[p.203]{bible} Schneider exhibits an example of a one parameter family of semialgebraic centrally symmetric convex bodies that are not zonoids (see Section~\ref{sec:Zonoids} for a definition of zonoids). Their support function is polynomial when restricted to the sphere. We will show how in that case Theorem~\ref{thm:supofsmooth} makes the computation of the fiber body relatively easy.

\begin{definition}
    Schneider's polynomial body is the convex body $\mathcal{S}_\alpha\in \KK(\R^3)$ whose support function is given by (see~\cite[p.203]{bible})
    \begin{equation}
        h_{\mathcal{S}_\alpha}(u)=\|u\|\left(1+\frac{\alpha}{2}\left(\frac{3(u_3)^2}{\|u\|^2}-1\right)\right)
    \end{equation}
    for $\alpha\in[-8/20,-5/20].$
\end{definition}
Let $\pi:=\langle e_1, \cdot \rangle:\R\oplus\R^2\to \R$ be the projection onto the first coordinate. We want to apply Theorem~\ref{thm:supofsmooth} to compute the support function of $\Sp{\mathcal{S}_\alpha}$. For the gradient we obtain:
\begin{equation}
    \nabla h_{\mathcal{S}_\alpha} (u)=\frac{1}{2\|u\|^3}\begin{pmatrix}
    -u_1\left((u_1)^2(\alpha-2)+ (u_2)^2(\alpha-2)+2(u_3)^2(2\alpha-1)\right)   \\ 
    -u_2\left((u_1)^2(\alpha-2)+ (u_2)^2(\alpha-2)+2(u_3)^2(2\alpha-1)\right)   \\
    \tfrac{u_3}{\|u\|^2}\left((u_1)^2(5\alpha+2)+ (u_2)^2(5\alpha+2)+2(u_3)^2(2\alpha+1)\right)  
    \end{pmatrix}.
\end{equation}
For $u=(0,u_2,u_3)$, the Jacobian is $J_{\psi_u}(t)=\frac{\dd}{\dd t}\left(\pi\circ\nabla h_{\mathcal{S}_\alpha}(t,u_2,u_3)\right)$, which gives
\begin{equation}
    J_{\psi_u}(t)=\frac{ t^2(-(u_2)^2(\alpha-2)+(u_3)^2(5\alpha+2))-\|u\|^2((u_2)^2(\alpha-2)+2(u_3)^2(2\alpha-1)) }{ 2(t^2+\|u\|^2)^\frac{5}{2} }.
\end{equation}
Substituting in~\eqref{eq:supofsmooth}, we integrate $\langle u, \nabla h_{\mathcal{S}_\alpha} (t,u_2,u_3)\rangle J_{\psi_u}(t)$ and get the support function of the fiber body (see Figure~\ref{fig:fiberofpoly}) which is again polynomial:
\begin{equation}\label{eq:suppofpoly}
    h_{\Sp{\mathcal{S}_\alpha}}(u)=\frac{\pi}{64\|u\|^3} \left({8(\alpha-2)(u_2)^4-8(\alpha^2+2\alpha-8)(u_2)^2(u_3)^2+(-25\alpha^2+16\alpha+32)(u_3)^4}\right).
\end{equation}

\begin{figure}
    \centering
    \includegraphics[width=0.4\textwidth]{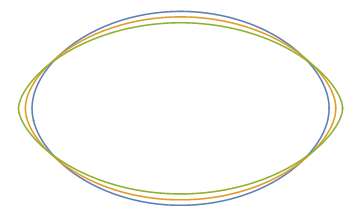}
    \caption{Fiber body of Schneider's polynomial body for $\alpha=-i/20$ with $i=5,6$ and $7$}
    \label{fig:fiberofpoly}
\end{figure}
\section{Zonoids}\label{sec:Zonoids}
In this section, we focus on the class of \emph{zonoids}. Let us first recall some definitions and introduce some notation. For more details we refer to~\cite[Section~$3.5$]{bible}. 

We will use the following notation for centered segments: for any $x\in\R^{n+m}$ we write
    \begin{equation}\label{eq:defseg}
        \seg{x}:=\frac{1}{2}\left[-x,x\right].
    \end{equation}

\begin{definition}
    A convex body $K\in\KK(\R^{n+m})$ is called a \emph{zonotope} if there exist $x_1,\ldots x_N\in\R^{n+m}$ such that, with the notation introduced above, $K=\seg{x_1}+\cdots+\seg{x_N}$. A \emph{zonoid} is a limit (in the Hausdorff distance) of zonotopes. The space of zonoids of $\R^{n+m}$ will be denoted by $\ZZo(\R^{n+m})$.
\end{definition}

\begin{remark}
    It follows immediately from the definition that all zonoids are centrally symmetric centered in the origin, i.e. if $K\in \ZZo(\R^{n+m})$ then $(-1)K=K$. In general the definition of zonoids may also include translations of such bodies. The elements of $\ZZo(\R^{n+m})$ are then called \emph{centered} zonoids. For simplicity here we chose to omit the term ``centered''.
\end{remark}

We introduce the approach of Vitale in~\cite{vitale} using random vectors. The following is~\cite[Theorem~$3.1$]{vitale} rewritten in our context.

\begin{proposition}\label{prop:VitaleZon}
    A convex body $K\in\KK(\R^{n+m})$ is a zonoid if and only if there is a random vector $X\in \R^{n+m}$ with $\EE\|X\|<\infty$ such that for all $u\in\R^{n+m}$
    \begin{equation}\label{eq:hofVitZon}
        h_K(u)=\frac{1}{2} \EE\left|\langle u, X\rangle\right|.
    \end{equation}
    We call such a zonoid the \emph{Vitale zonoid} associated to the random vector $X$, and denote it by $K_0(X)$.
\end{proposition}

\subsection{The fiber body of a zonoid}\label{sec:fiber_of_zonoid}

We now show that the fiber body of a zonoid is a zonoid and give a formula to compute it in Theorem~\ref{thm:Fiberofzonoids}. Let us first introduce some of the tools used by Esterov in~\cite{esterovmixedfiber}.
\begin{definition}
    For any $u\in W$ define $T_u:=Id_V\oplus \langle u, \cdot\rangle : V\oplus W \to V \oplus \R$.
\end{definition}
\begin{definition}
    Let $C\in\KK(V\oplus \R)$. The \emph{shadow volume} $V_+(C)$ of $C$ is defined to be the integral of the maximal function on $\pi(C)\subset V$ such that its graph is contained in C, i.e. 
    \begin{equation}
        V_+(C)=\int_{\pi(C)}\varphi(x)\dd x,
    \end{equation} where $\varphi(x)=\sup\set{t}{(x,t)\in C}$. In particular if $(-1)C=C$, then the shadow volume is $V_+(C)=\tfrac{1}{2}\Vol_{n+1}(C)$.
\end{definition}

The \emph{shadow volume} can then be used to express the support function of the fiber body.

\begin{lemma}
    For $u\in W$  and $K\in\KK(\R^{n+m})$, we have
    \begin{equation}
        h_{\Sigma_\pi K}(u)=V_+\left(T_u(K)\right).
    \end{equation}
    In particular if $(-1)K=K$,
    \begin{equation}\label{eq:suppShadow2}
        h_{\Sigma_\pi K}(u)=\frac{1}{2}\Vol_{n+1}\left(T_u(K)\right).
    \end{equation}
    \begin{proof}
        We also denote by $\pi: V\oplus \R\to V$ the projection onto $V$. The shadow volume is the integral on $\pi(T_u(K))=\pi(K)$ of the function $\varphi(x)=\sup\set{t}{(x,t)\in T_u(K)}=\sup\set{\langle u, y \rangle}{(x,y)\in K}=h_{K_x}(u)$. Thus the result follows from Proposition~\ref{prop:supportaverage}.
    \end{proof}
\end{lemma}

\begin{remark}\label{rmk:projectionbody}
        Note that if $m=2$ then $T_u$ is the projection onto the hyperplane spanned by $V$ and $u$. In that case~\eqref{eq:suppShadow2} is the formula for the support function of the \emph{projection body} $\Pi K$ of $K$ at $Ju$, where $J$ is a rotation by $\pi/2$ in $W$, see~\cite[Section~10.9]{bible}. Thus in that case, $\Sigma_\pi K$ is the projection of $\Pi K$ onto $W$ rotated by $\pi/2$. 
\end{remark}

We will show that the mixed fiber body of zonoids comes from a multilinear map defined directly on the vector spaces.

\begin{definition} We define the following (completely skew-symmetric) multilinear map:
    \begin{align}
        F_\pi:(V\oplus W)^{n+1}             &\to W  \\
        \left((x_1,y_1),\ldots,(x_{n+1},y_{n+1})\right)    &\mapsto \frac{1}{(n+1)!}\sum_{i=1}^{n+1}(-1)^{n+1-i} (x_1\wedge\cdots\wedge \widehat{x_i} \wedge \cdots \wedge x_{n+1}) y_{i}
    \end{align} 
    where $x_1\wedge\cdots\wedge \widehat{x_i} \wedge \cdots \wedge x_{n+1}$ denotes the determinant of the chosen vectors omitting $x_i$.
\end{definition}

We are now able to prove the main result of this section, here stated in the language of the Vitale zonoids introduced in Proposition~\ref{prop:VitaleZon}.

\begin{theorem}\label{thm:Fiberofzonoids}
    The fiber body of a zonoid is a zonoid. Moreover, if $X\in \R^{n+m}$ is a random vector such that $\EE\|X\|<\infty$ and $K:=K_0(X)$ is the associated Vitale zonoid, then
    \begin{equation}\label{eq:FormulaZon}
        \Sigma_\pi K=K_0(F_\pi(X_1,\ldots,X_{n+1}))
    \end{equation}
    where $X_1,\ldots,X_{n+1}\in \R^{n+m}$ are i.i.d. copies of $X$.
    In other words, the support function of the fiber body $\Sigma_\pi K$ is given for all $u\in W$ by
    \begin{equation}\label{eq:FormulaZonsupp}
        h_{\Sigma_\pi K}(u)=\frac{1}{2}\EE|\langle u, Y\rangle|
    \end{equation}
    where $Y\in W$ is the random vector defined by $Y:=F_\pi(X_1,\ldots,X_{n+1})$.
    \begin{proof}
        Suppose that $K=K_0(X)$ and let $u\in W$. Note that by~\eqref{eq:hofVitZon} and Proposition~\ref{prop:propertieshK}--$\mathit{(ii)}$, $T_u(K)=K_0\left(T_u(X_1)\right)$. Thus by~\eqref{eq:suppShadow2} and \cite[Theorem~$3.2$]{vitale} we get 
        \begin{equation}\label{eq:spkinproof}
            h_{\Sp{K}}(u)=\tfrac{1}{2}\Vol\left(K_0(T_u(X))\right)=\tfrac{1}{2}\tfrac{1}{(n+1)!}\EE|T_u(X_1)\wedge\cdots\wedge T_u(X_{n+1})|
        \end{equation}
        where $X_1,\ldots,X_{n+1}\in \R^{n+m}$ are iid copies of $X$.
        
        Now let us write $X_i:=(\alpha_i,\beta_i)$ with $\alpha_i\in V$ and $\beta_i\in W$. Then
        \begin{align}
           \left|T_u(X_1)\wedge\cdots\wedge T_u(X_{n+1}) \right| &= \left|\left(\alpha_1,\langle u, \beta_1\rangle\right)\wedge \cdots \wedge \left(\alpha_{n+1},\langle u, \beta_{n+1}\rangle\right) \right|     \\
            &=\left| \sum_{i=1}^{n+1}(-1)^{n+1-i} (\alpha_1\wedge\cdots\wedge \widehat{\alpha_i} \wedge \cdots \wedge \alpha_{n+1}) \langle u,\beta_{i}\rangle\right|   \\
            &=\left|\langle u,(n+1)! F_\pi\left((\alpha_1,\beta_1),\ldots,(\alpha_{n+1},\beta_{n+1})\right) \rangle \right|.
    \end{align}
    Reintroducing this in~\eqref{eq:spkinproof} we obtain~\eqref{eq:FormulaZonsupp}.
        
    \end{proof}
\end{theorem}

This allows to generalize~\cite[Theorem~$4.1$]{fiberpolytopes} for all zonotopes. 
\begin{corollary}\label{cor:fiberofzonotopes}
    For  all $z_1,\ldots, z_{N}\in\R^{n+m}$, the fiber body of the zonotope $\sum_{i=1}^N\seg{z_i}$ is the zonotope given by
    \begin{equation}\label{fiber_zonotope}
        \Sigma_\pi\left(\sum_{i=1}^N\seg{z_i}\right)=(n+1)!\sum_{1\leq i_1<\cdots<i_{n+1}\leq N} \seg{F_\pi(z_{i_1},\ldots, z_{i_{n+1}})}
    \end{equation}
     where we used the notation of~\eqref{eq:defseg}, writing $\seg{x}$ for the segment $[-x/2,x/2]$.
    \begin{proof}
       We apply Theorem~\ref{thm:Fiberofzonoids} to the discrete random vector $X$, that is equal to $N z_i$ with probability $1/N$ for all $i=1,\ldots,N$. In that case one can check from~\eqref{eq:hofVitZon} that the Vitale zonoid $K_0(X)$ is precisely the zonotope $\sum_{i=1}^N\seg{z_i}$, and the result follows from~\eqref{eq:FormulaZonsupp}.
    \end{proof}
\end{corollary}
An implementation of formula \eqref{fiber_zonotope} for \texttt{OSCAR 0.8.2-DEV} \cite{OSCAR} and \texttt{SageMath 9.2} \cite{sagemath} is available at \url{ https://mathrepo.mis.mpg.de/FiberZonotopes}.

Esterov shows in~\cite{esterovmixedfiber} that the map $\Sp:\KK(\R^{n+m})\to \KK(W)$ comes from another map, which is (Minkowski) multilinear in each variable: the \emph{mixed fiber body}. The following is~\cite[Theorem~$1.2$]{esterovmixedfiber}.

\begin{proposition}
    There is a unique symmetric multilinear map 
    \begin{equation}
        \MSp:\left(\KK(\R^{n+m})\right)^{n+1}\to \KK(W)
    \end{equation} such that for all $K\in\KK(\R^{n+m})$,  $\MSp(K,\ldots,K)=\Sigma_\pi(K)$.
\end{proposition}

Once its existence is proved, one can see that the mixed fiber body $\MSp(K_1,\ldots,K_{n+1})$ is the coefficient of $t_1\cdot\ldots\cdot t_{n+1}$, divided by $(n+1)!$, in the expansion of $\Sp{\left(t_1 K_1+\cdots+t_{n+1}K_{n+1}\right)}$. Using this \emph{polarization formula}, one can deduce from Theorem~\ref{thm:Fiberofzonoids} a similar statement for the mixed fiber body of zonoids.

\begin{proposition}\label{prop:mixedfiberofzon}
    The mixed fiber body of zonoids is a zonoid. Moreover, if $X_1,\ldots,X_{n+1}\in \R^{n+m}$ are independent (not necessarily identically distributed) random vectors such that $\EE\|X_i\Vert$ is finite, and $K_i:=K_0(X_i)$ are the associated Vitale zonoids, then 
    \begin{equation}
        \MSp(K_1,\ldots,K_{n+1})=K_0(F_\pi(X_1,\ldots,X_{n+1})).
    \end{equation}
    \begin{proof}
     Let us show the case of $n+1=2$ variables. The general case is done in a similar way. Let $\tilde{X}:=t_1\alpha 2 X_1+ t_2 (1-\alpha) 2 X_2$ where $\alpha$ is a Bernoulli random variable of parameter $1/2$ independent of $X_1$ and $X_2$. Using~\eqref{eq:hofVitZon}, one can check that $K_0(\tilde{X})=t_1 K_1+t_2 K_2.$
     Now let $Y_1$ (respectively $Y_2$) be an i.i.d. copy of $X_1$ (respectively $X_2$) independent of all the other variables. Define $\tilde{Y}:=t_1\beta 2 Y_1+ t_2 (1-\beta) 2 Y_2$ where $\beta$ is a Bernoulli random variable of parameter $1/2$ independent of all the other variables. By Theorem~\ref{thm:Fiberofzonoids} we have that $\Sp(t_1K_1+t_2K_2)=K_0(F_\pi(\tilde{X},\tilde{Y})).$ By~\eqref{eq:hofVitZon}, using the independence assumptions, it can be deduced that for all $t_1,t_2\geq 0$
     \begin{equation}
         h_{K_0(F_\pi(\tilde{X},\tilde{Y}))}=t_1^2 h_{\Sp K_1}+t_2^2 h_{\Sp K_2}+t_1t_2 (h_{K_0(F_\pi(X_1,Y_2))}+h_{K_0(F_\pi(X_2,Y_1))}).
     \end{equation}
     The claim follows from the fact that $K_0(F_\pi(X_1,Y_2))=K_0(F_\pi(X_2,Y_1))=K_0(F_\pi(X_1,X_2))$.
    \end{proof}
\end{proposition}

\subsection{Discotopes}
In this section, we investigate the fiber bodies of finite Minkowski sums of discs in $\R^3$, called \emph{discotopes}. They also appear in the literature, see~\cite{sanyal_discotopes} for example. Discotopes are zonoids (because discs are zonoids see Lemma~\ref{lem:discotopeeq} below) that are neither polytopes nor curved (see Section~\ref{sec:Smooth}) but still have simple combinatorial properties and a simple support function. For a deep analysis of this family of zonoids, we refer to \cite{GesMer:Discotopes}. We will see how in this case formula~\eqref{eq:FormulaZonsupp} can be useful to compute the fiber body.

\begin{definition}
    Let $v\in \R^3$, we denote by $D_v$ the disc in $v^\perp$ centered at 0 of radius $\|v\|$.
\end{definition}

\begin{lemma}\label{lem:discotopeeq}
Discs are zonoids. If $a,b$ is an orthonormal basis of $v^\perp$, we define the random vector $\sigma (\theta):= \|v\| (\cos(\theta)a+\sin(\theta)b)$ with $\theta\in[0,2\pi]$ uniformly distributed. Then we have 
\begin{equation}\label{eq:DvasKofs}
    D_v=\pi\cdot K_0\left(\sigma (\theta)\right)
\end{equation} where we recall the definition of the Vitale Zonoid associated to a random vector in Proposition~\ref{prop:VitaleZon}. In other words we have:
\begin{equation}\label{eq:DiscRandomVar}
    h_{D_v}(u)=\|v\|\sqrt{\langle u,a\rangle^2+\langle u,b\rangle^2}=\frac{\pi}{2}\EE|\langle u, \sigma (\theta)\rangle|.
\end{equation}
\end{lemma}
\begin{proof}
    Consider the zonoid $K_0\left(\sigma (\theta)\right)$. We will prove that it is a disc contained in $v^\perp$ centered at $0$ of radius $\|v\|/\pi$.
    
    First of all, since $\sigma (\theta)\in v^\perp$ almost surely, we have $h_{K_0\left(\sigma (\theta)\right)}(\pm v)=0$. Thus $K_0\left(\sigma (\theta)\right)$ is contained in the plane $v^\perp$. Moreover, let $O(v^\perp)$ denote the stabilizer of $v$ in the orthogonal group $O(3)$. The zonoid $K_0\left(\sigma (\theta)\right)$ is invariant under the action of $O(v^\perp)$ thus it is a disc centered at $0$. To compute its radius it is enough to compute the support function at one point: $h_{K_0\left(\sigma (\theta)\right)}(a_1)=\|v\|\cdot\EE|\cos(\theta)|=\|v\|/\pi$ and this concludes the proof.
\end{proof}

\begin{remark}
        Note that the law of the random vector $\sigma(\theta)$ does not depend on the choice of the orthonormal basis $a,b$. It only depends on the line spanned by $v$ and the norm $\|v\|$.
\end{remark}

\begin{definition}
    A convex body $K\subset \R^3$ is called a \emph{discotope} if it can be expressed as a finite Minkowski sum of discs, i.e. if there exist $v_1,\ldots, v_N\in\R^3$, 
    such that $K=D_{v_1}+\cdots+D_{v_N}$. In particular discotopes are zonoids. Moreover we can and will assume without loss of generality that
    \begin{equation}
        \frac{v_i}{\|v_i\|} \neq \pm \frac{v_j}{\|v_j\|} \qquad \hbox{for } i\neq j.
    \end{equation}
\end{definition}

What is the shape of a discotope? In order to answer this question we are going to study the boundary structure of such a convex body, when $N\geq 2$. 
\begin{lemma}\label{lem:boundary_discs}
    Consider the discotope $K = D_{v_1} + \ldots + D_{v_N}$, fix $q\in \partial (D_{v_2}+\ldots + D_{v_N})$ and take the Minkowski sum $D_{v_1}+\{q\}$.
    Then such disc is part of the boundary of the discotope if and only if 
    \begin{equation}\label{eq:cond_qmax}
        \langle q , v_1 \rangle = \pm \max\set{ \langle \tilde{q} , v_1 \rangle}{\tilde{q}\in D_{v_2}+\ldots+D_{v_N}}.
    \end{equation}
\end{lemma}
\begin{proof}
    We do the proof for $N=2$; the general case is then given by a straightforward induction. 
    Let $r:S^2 \to \R_{\geq 0}$ be the radial function of the discotope, namely $r(x) := \max \set{\lambda \geq 0}{\lambda x \in K}$. A point $x \in \partial K$ if and only if $r\left(\frac{x}{\| x \|}\right) = \| x \|$. So we claim that for all $p\in D_{v_1}$
    \begin{equation}
        r\left( \frac{p+q}{\| p+q \|}\right) = \| p+q \| 
    \end{equation}
    where $q\in D_{v_2}$ satisfies $\langle q , v_1 \rangle = \pm \max\set{ \langle \tilde{q} , v_1 \rangle}{\tilde{q}\in D_{v_2}}$.
    Assume first that $q$ realizes the maximum. Let $ r\left( \frac{p+q}{\| p+q \|}\right) = \lambda$. Then we have:
    \begin{equation}
        \lambda  \left( \frac{p+q}{\| p+q \|}\right) = p' + q' \in \partial K
    \end{equation}
    for some $p'\in D_{v_1}$ and $q'\in D_{v_2}$. By taking the scalar product with $v_1$ we get:
    \begin{equation}
        \frac{\lambda}{\| p+q \|} \langle q, v_1 \rangle = \langle q' , v_1 \rangle \leq \langle q, v_1 \rangle
    \end{equation}
    therefore $\lambda \leq \| p+q \|$. Since $p+q$ is a point of $K$, $\lambda \geq \| p+q \|$ and the thesis follows.\\
    The other case where $q$ realizes the minimum is analogous.
\end{proof}

Since we assumed that all the $v_i$ are non colinear, for every $i$ there are exactly two $q_i$ that satisfy \eqref{eq:cond_qmax} that we will denote by $q_i^+$ and $q_i^-$ respectively. Lemma~\ref{lem:boundary_discs} then says that in the boundary of the discotope there are exactly $2N$ discs, namely
\begin{equation}
    D_{v_1} + \{q_1^+\},\; D_{v_1} + \{q_1^-\}, \;\ldots,\; D_{v_N} + \{q_N^+\},\; D_{v_N} + \{q_N^-\}.
\end{equation}
The rest of the boundary of the discotope is the open surface $\mathcal{S}:=\partial K \setminus \cup_{i=1}^N (D_{v_i}+\{q_i^{\pm}\})$ made of exposed points. Moreover we show in the next proposition that $\mathcal{S}$ has either one or two connected components. 

\begin{proposition}\label{prop:componentsS}
    Consider the discotope $K = D_{v_1} + \ldots + D_{v_N}$, then $\mathcal{S}$ has two connected components if and only if $v_1, \ldots, v_N$ lie all in the same plane. Otherwise it is connected and no two discs intersect.
\end{proposition}
\begin{proof}
    Assume first that $v_1, \ldots, v_N \in H$ where without loss of generality $H$ is the hyperplane defined by $\{z=0\}$, then we claim that all the discs in $\partial K$ meet on $H$ in a very precise configuration. Trivially the Minkowski sum $\left( D_{v_1} \cap H \right) + \ldots + \left( D_{v_N} \cap H \right)$ is contained in $K\cap H$. On the other hand let $p\in K\cap H$, then
    \begin{equation}
        p = (\alpha_1,\beta_1,\gamma_1) + \ldots + (\alpha_N,\beta_N,\gamma_N)
    \end{equation}
    where $(\alpha_i,\beta_i,\gamma_i)\in D_{v_i}$ and $\sum \gamma_i = 0$. But because $v_i \in H$, then also $(\alpha_i,\beta_i,0)\in D_{v_i}$ and so we can write $p$ as
    \begin{equation}
        p = (\alpha_1,\beta_1,0) + \ldots + (\alpha_N,\beta_N,0)
    \end{equation}
    hence $p\in \left( D_{v_1} \cap H \right) + \ldots + \left( D_{v_N} \cap H \right)$. This implies that $K\cap H$ is a $2$--dimensional zonotope with $2N$ edges, as in Figure~\ref{fig:touching_discs}; its vertices are exactly the points of intersection of the discs in the boundary. Hence the boundary discs divide $\mathcal{S}$ in exactly $2$ connected components.
    \begin{figure}
        \centering
        \def\svgwidth{0.7\textwidth}
\begingroup%
  \makeatletter%
  \providecommand\color[2][]{%
    \errmessage{(Inkscape) Color is used for the text in Inkscape, but the package 'color.sty' is not loaded}%
    \renewcommand\color[2][]{}%
  }%
  \providecommand\transparent[1]{%
    \errmessage{(Inkscape) Transparency is used (non-zero) for the text in Inkscape, but the package 'transparent.sty' is not loaded}%
    \renewcommand\transparent[1]{}%
  }%
  \providecommand\rotatebox[2]{#2}%
  \newcommand*\fsize{\dimexpr\f@size pt\relax}%
  \newcommand*\lineheight[1]{\fontsize{\fsize}{#1\fsize}\selectfont}%
  \ifx\svgwidth\undefined%
    \setlength{\unitlength}{516.26274049bp}%
    \ifx\svgscale\undefined%
      \relax%
    \else%
      \setlength{\unitlength}{\unitlength * \real{\svgscale}}%
    \fi%
  \else%
    \setlength{\unitlength}{\svgwidth}%
  \fi%
  \global\let\svgwidth\undefined%
  \global\let\svgscale\undefined%
  \makeatother%
  \begin{picture}(1,0.61285911)%
    \lineheight{1}%
    \setlength\tabcolsep{0pt}%
    \put(0,0){\includegraphics[width=\unitlength,page=1]{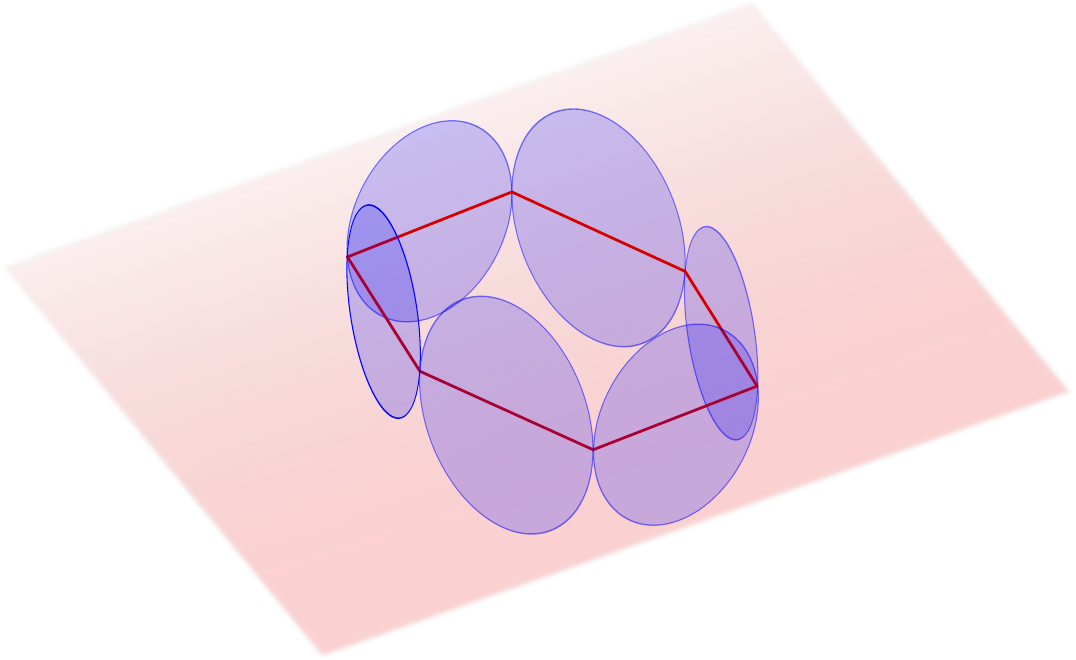}}%
    \put(0.29739618,0.04997373){\color[rgb]{0,0,0}\makebox(0,0)[lt]{\lineheight{1.25}\smash{\begin{tabular}[t]{l}$H$\end{tabular}}}}%
  \end{picture}%
\endgroup%

        \caption{The $6$ blue discs are part of the boundary of the discotope $K = D_{v_1} + D_{v_2} + D_{v_3}$, where $v_1,v_2,v_3$ belong to the red-shaded hyperplane H. It separates the two connected components of $\mathcal{S}$. In particular the intersection $\partial K\cap H$ is the red hexagon. }
        \label{fig:touching_discs}
    \end{figure}
    
    For the converse notice that if there are two connected components, then at least two boundary discs must intersect. Without loss of generality assume that there is an intersection point $p$ between a copy of $D_{v_1}$ and a copy of $D_{v_2}$ and consider the plane $H = \text{span}(v_1,v_2)$. Let $\pi (K)$ be the projection of the discotope on $H$; clearly $\pi(p)\in \partial \pi(K)$ is a vertex. Then for $u\in S^1\hookrightarrow H$
    \begin{align}
        h_{\pi(K)}(u) &= h_{D_{v_1}}(u) + \ldots + h_{D_{v_N}}(u) \\
                      &\overset{\eqref{eq:DiscRandomVar}}{=} \sum_{i=1}^N \|v_i\|\sqrt{\langle u,a_i\rangle^2+\langle u,b_i\rangle^2} \\
                      &= \sum_{i=1}^N \|v_i\|\sqrt{\langle u,\pi(a_i)\rangle^2+\langle u,\pi(b_i)\rangle^2}
    \end{align}
    where $\{\frac{v_i}{\|v_i\|},a_i,b_i\}$ is an orthonormal basis for every $i$. There are two possibilities now: either $\pi(a_i)$ and $\pi(b_i)$ are linearly independent, or they are linearly dependent and possibly zero. The latter case corresponds to discs such that $v_i\in H$, and the summand above becomes linear. So, up to relabeling, we can rewrite the support function splitting these cases:
    \begin{equation}
        h_{\pi(K)}(u) = \sum_{i=1}^k |\langle u,\alpha_i \rangle | +  \sum_{j=k+1}^N \|v_j\|\sqrt{\langle u,\pi(a_j)\rangle^2+\langle u,\pi(b_j)\rangle^2} 
    \end{equation}
    for some $\alpha_i \in \R$ and $2\leq k \leq N$. Therefore $\pi(K)$ is the Minkowski sum of $k$ line segments and $N-k$ ellipses. The boundary contains a vertex if and only if there are no ellipses in the sum, hence $k=N$ i.e. $v_i \in H$ for every $i$.
\end{proof}

\begin{remark}
    The previous result can be interpreted with the notion of \emph{patches}. These geometric objects have been first introduced in \cite{sturmfels_patches} and allow to subdivide the boundary of a convex body. Accordingly to their definition, in the discotope we find $2N$ $2$--patches, corresponding to the boundary discs, and either one ore two $0$--patches when $\mathcal{S}$ has one or two connected components respectively.
    Recently Plaumann, Sinn and Wesner \cite{sinn_patches} refined the definition of patches for a semialgebraic convex body. In this setting it is more subtle to count the number of patches of our discotopes, because this requires the knowledge of the number of irreducible components of $\mathcal{S}$.
\end{remark}

\subsection{A case study: the dice}

\begin{definition}
    Let $e_1,e_2,e_3$ be the standard basis of $\R^3$ and let $D_i:=D_{e_i}$. We define the \emph{dice} to be the discotope $\mathscr{D}:=D_1+D_2+D_3$. See Figure~\ref{fig:Dice}.
\end{definition}
The boundary of the dice consists of $6$ two--dimensional discs of radius $1$, lying in the center of the facets of the cube $[-2,2]^3$, and a connected surface. The latter is the zero locus of the polynomial of degree 24:
\begin{equation}
    \varphi(x,y,z) = x^{24}+4 x^{22} y^2+2 x^{20} y^4+ \ldots +728 z^4-160 x^2-160 y^2-160 z^2+16
\end{equation}
which is too long to fit in a page (it is made of $91+78+66+55+45+36+28+21+15+10+6+3+1=455$ monomials, here distinguished by their degree).

Consider the projection $\pi:=\langle e_1, \cdot\rangle :\R\oplus \R^2\to \R$. Even in this simple example the fibers of the dice under this projection can be tricky to describe. However using the formula for zonoids one can compute explicitly the fiber body (see Figure~\ref{fig:FiberOfDice}).
\begin{figure}[ht]
    \begin{subfigure}{0.49\textwidth}
    \centering
    \includegraphics[width=0.6\textwidth]{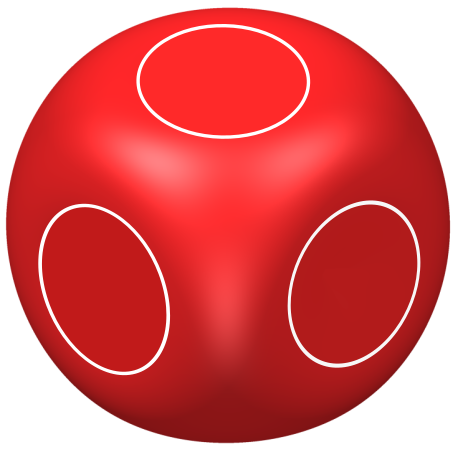}
    \caption{}
    \label{fig:Dice}
    \end{subfigure}
    \begin{subfigure}{0.49\textwidth}
    \centering
    \includegraphics[width=0.5\textwidth]{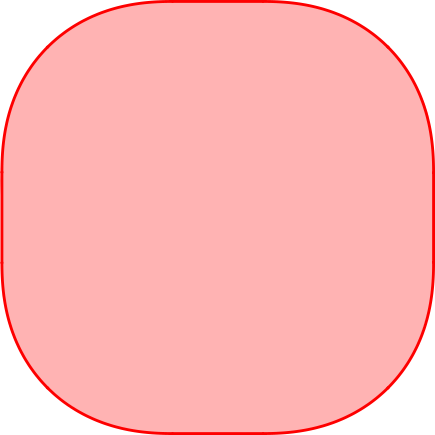}
    \caption{}
    \label{fig:FiberOfDice}
    \end{subfigure}
    \caption{Left: the dice. Right: its fiber body.}
\end{figure}

\begin{proposition}
    With respect to this projection $\pi$, the fiber body of $\mathscr{D}$ is 
    \begin{equation}
        \Sigma_\pi(\mathscr{D})= D_1+\frac{\pi}{4} \left(\seg{e_2}+\seg{e_3}\right)+ \frac{1}{2}\Lambda
    \end{equation}
    where $\Lambda$ is the convex body whose support function is given by
    \begin{equation}
        h_\Lambda(u_2,u_3)=\frac{1}{2}\int_0^{\pi}\sqrt{\cos(\theta)^2 \left(u_2\right)^2+\sin(\theta)^2 \left(u_3\right)^2} \ \dd \theta
    \end{equation}
    and where we recall the notation~\eqref{eq:defseg} for segments.
    \begin{proof}
     First of all let us note that by expanding the mixed fiber body $\MSp(\mathscr{D},\mathscr{D})$ we have 
    \begin{equation}\label{eq:sumfibdice}
        \Sigma_\pi(\mathscr{D})=\Sigma_\pi(D_1)+\Sigma_\pi(D_2)+\Sigma_\pi(D_3)+2\left( \MSp(D_1,D_2)+\MSp(D_1,D_3)+\MSp(D_2,D_3)\right).
    \end{equation}
    Now let $\sigma_1(\theta):=(0,\cos(\theta),\sin(\theta))$, $\sigma_2(\theta):=(\cos(\theta),0,\sin(\theta))$ and $\sigma_3(\theta):=(\cos(\theta),\sin(\theta),0)$ in such a way that $h_{D_i}(u)=\frac{\pi}{2}\EE|\langle u, \sigma_i(\theta)\rangle|$.
    
    We then want to use Theorem~\ref{thm:Fiberofzonoids} and Proposition~\ref{prop:mixedfiberofzon} to compute all the summands of the expansion of $\Sigma_\pi(\mathscr{D})$. Using~\eqref{eq:DvasKofs} we have that $\MSp(D_i,D_j)=\pi^2 K_0 (F_\pi(\sigma_i(\theta),\sigma_j(\phi))$ with $\theta,\phi\in [0,2\pi]$ uniform and independent. In our case, $F_\pi(x,y)=(x_1y_2-y_1x_2, x_1y_3-y_1x_3)/2$. We obtain
    \begin{align}
        F_\pi(\sigma_1(\theta),\sigma_1(\phi)) =0,     \;\: &F_\pi(\sigma_2(\theta),\sigma_2(\phi))    =\frac{1}{2}(0,\sin(\phi-\theta)),         \\
        F_\pi(\sigma_3(\theta),\sigma_3(\phi))=\frac{1}{2}(\sin(\phi-\theta),0), \;\:  &F_\pi(\sigma_1(\theta),\sigma_2(\phi))    =\frac{-\cos(\phi)}{2}(\cos(\theta),\sin(\theta)), \\
        F_\pi(\sigma_1(\theta),\sigma_3(\phi))    =\frac{-\cos(\phi)}{2}(\cos(\theta),\sin(\theta)),    \;\:     &F_\pi(\sigma_2(\theta),\sigma_3(\phi))    =\frac{1}{2}(\cos(\theta)\sin(\phi),\sin(\theta)\cos(\phi)).   
    \end{align}
  Computing the support function $h_{\pi^2K_0 (F_\pi(\sigma_i(\theta),\sigma_j(\phi)))}=(\pi^2/2)\EE|\langle u, F_\pi(\sigma_i(\theta),\sigma_j(\phi))\rangle|$ and using that $\EE|\cos(\phi)|=2/\pi$, we get 
  \begin{align}
        &\Sigma_\pi(D_1)=0; \quad \Sigma_\pi (D_2)=\frac{\pi}{4}\ \seg{e_2}; \quad \Sigma_\pi (D_3)=\frac{\pi}{4}\ \seg{e_3} ;   \\
        &\MSp(D_1,D_2)=\MSp(D_1,D_3)=\frac{1}{4} D_1
    \end{align}
     It only remains to compute $\MSp(D_2,D_3)$. We have 
    \begin{equation}
        h_{\MSp(D_2,D_3)}(u)=\frac{1}{2}\left(\frac{\pi}{2}\right)^2\EE|\langle u, F_\pi(\sigma_2(\theta),\sigma_3(\phi))\rangle|=\frac{\pi^2}{16}\EE|u_2 \cos(\theta)\sin(\phi)+u_3\sin(\theta)\cos(\phi)|.
    \end{equation}
    We use then the independence of $\theta$ and $\phi$ and~\eqref{eq:DiscRandomVar} to find 
    \begin{equation}
        h_{\MSp(D_2,D_3)}(u)=\frac{\pi}{8}\EE\sqrt{\cos(\theta)^2 \left(u_2\right)^2+\sin(\theta)^2 \left(u_3\right)^2}=\frac{1}{4}h_\Lambda(u)
    \end{equation}
   Puting back together everything we obtain the result.
    \end{proof}
\end{proposition}

    \begin{remark}
        It is worth noticing that the convex body $\Lambda$ also appears, up to a multiple, in~\cite[Section~$5.1$]{PSC} where it is called $D(2)$, with no apparent link to fiber bodies.
        In the case where $u_2\neq 0$ we have 
        \begin{equation}
            h_\Lambda(u)= |u_2| E\left(\sqrt{1-\left(\frac{u_3}{u_2}\right)^2}\right)
        \end{equation}
        where $E(s)=\int_0^{\pi/2}\sqrt{1-s^2\sin(\theta)^2}\dd \theta$ is the complete elliptic integral of the second kind. This function is not semialgebraic thus the example of the dice shows that the fiber body of a semialgebraic convex body is not necessarily semialgebraic. However $E$ is holonomic. This suggests that the curved assumption in Corollary~\ref{cor:holonomiccurved} may not be needed.
    \end{remark}

\bibliographystyle{alpha}
\bibliography{biblio}

~\\

\noindent{\bf Authors' addresses:} 
\smallskip

\small 

\noindent {\bf L\'eo Mathis,} \\
Goethe-Universit\"at Frankfurt, \\
Robert-Mayer-Strasse 10, D-60325 Frankfurt am Main, Germany \\
\hfill {\tt mathis@mathematik.uni-frankfurt}
\smallskip

\noindent {\bf Chiara Meroni, }\\
Max Planck Institute for Mathematics in the Sciences, \\
Inselstrasse 22, 04103 Leipzig, Germany \\
\hfill {\tt chiara.meroni@mis.mpg.de}

\end{document}